\documentclass[15pt]{amsart}
\usepackage{amsmath}
\usepackage{amssymb}
\usepackage[all]{xy}
\usepackage{graphicx}
\usepackage{mathrsfs}

\usepackage{graphicx}
\usepackage[all]{xy}
\SelectTips{cm}{}

\usepackage[dvips]{epsfig}

\usepackage{pinlabel}

\numberwithin{equation}{section}

\usepackage{url}
\usepackage{tikz}
\usetikzlibrary{positioning}
\usetikzlibrary{calc}

\usepackage[a4paper, total={6.5in,9in}]{geometry}

\usetikzlibrary{decorations.pathreplacing}

\usepackage[latin1]{inputenc}
\usepackage{xspace,amssymb,amsfonts,euscript}
\usepackage{amsthm,amsmath}
\usepackage{euscript}
\input xy \xyoption {all}

\usepackage{tikz}
\usepackage{wrapfig}

\RequirePackage{color}
\definecolor{myred}{rgb}{0.75,0,0}
\definecolor{mygreen}{rgb}{0,0.5,0}
\definecolor{myblue}{rgb}{0,0,0.65}

\RequirePackage{ifpdf}
\ifpdf
  \IfFileExists{pdfsync.sty}{\RequirePackage{pdfsync}}{}
  \RequirePackage[pdftex,
   colorlinks = true,
   urlcolor = myblue, 
   citecolor = mygreen, 
   linkcolor = myred, 
   pagebackref,
   bookmarksopen=true]{hyperref}
\else
  \RequirePackage[hypertex]{hyperref}
\fi

\RequirePackage{ae, aecompl, aeguill} 

    \def\QM{{\mathbb{Q}}}

    \def\ZM{{\mathbb{Z}}}

\def\HC{{\mathbf H}}    \def\HC{{\mathcal{H}}}

    \def\MC{{\mathcal{M}}}

\def\b{\beta}

\newcommand{\nc}{\newcommand} \newcommand{\renc}{\renewcommand}

\newcommand{\rdots}{\mathinner{ \mkern1mu\raise1pt\hbox{.}
    \mkern2mu\raise4pt\hbox{.}
    \mkern2mu\raise7pt\vbox{\kern7pt\hbox{.}}\mkern1mu}}

\def\un{\underline}

\def\p{{}^p}

\def\to{\rightarrow}

\def\longto{\longrightarrow}

\def\onto{\twoheadrightarrow}
\nc{\triright}{\stackrel{[1]}{\to}}
\nc{\longtriright}{\stackrel{[1]}{\longto}}

\nc{\Hb}{H^\bullet}

\nc{\Br}{\mathcal{\color{red}b}}
\nc{\HotRR}{{}_R\mathcal{K}_R}
\nc{\HotR}{\mathcal{K}_R}
\nc{\excise}[1]{}
\nc{\defect}{\text{df}}
\nc{\h}[1]{\underline{H}_{#1}}

\nc{\Ga}{\mathbb{G}_a} 
\nc{\Gm}{\mathbb{G}_m} 

\nc{\Perv}{{\mathbf{P}}}

\nc{\IH}{{\mathrm{IH}}}

\nc{\ic}{\mathbf{IC}}

\nc{\gl}{{\mathfrak{gl}}}
\renc{\sl}{{\mathfrak{sl}}}
\renc{\sp}{{\mathfrak{sp}}}

\renc{\Im}{\textrm{Im}}

\nc{\HCM}{H^{BM}}

 \DeclareMathOperator{\Hom}{Hom}
\DeclareMathOperator{\gHom}{\Hom^\bullet}
 \DeclareMathOperator{\ch}{ch}

\DeclareMathOperator{\End}{End} 

\DeclareMathOperator{\id}{id}

\newtheorem{thm}{Theorem}[section]
\newtheorem{lem}[thm]{Lemma}

\theoremstyle{definition}

\theoremstyle{remark}

\newcommand{\into}{\hookrightarrow}

\def\iff{\Leftrightarrow}

\nc{\simto}{\stackrel{\sim}{\to}}

\nc{\SD}{\mathcal{H}_{\mathrm{BS}}}

\title{A non-perverse  Soergel bimodule in type $A$}

 \author{Nicolas Libedinsky and Geordie Williamson}

\begin{document}


\begin{abstract}
A basic question concerning indecomposable Soergel bimodules is to understand their
endomorphism rings. In characteristic zero all
degree-zero endomorphisms are isomorphisms (a fact proved
by Elias and the second author) which implies the Kazhdan-Lusztig
conjectures. More recently, many examples in
positive characteristic have been discovered with larger degree zero
endomorphisms. These give counter-examples to expected bounds
in Lusztig's conjecture. Here we prove the existence of
indecomposable Soergel bimodules in type $A$ having non-zero
endomorphisms of negative degree. This gives the existence of a
non-perverse parity sheaf in type \emph{A}.
\end{abstract}

\maketitle

\section{Introduction}

Kazhdan-Lusztig polynomials play a central role in highest weight
representation theory. It is gradually becoming clear that in
modular (i.e. characteristic $p$) representation theory a similarly
central  role is played by $p$-Kazhdan-Lusztig polynomials
\cite{RW,JW,WExplosion,WTakagi,AMRW2}. Just as
Kazhdan-Lusztig polynomials describe the stalks of intersection
cohomology complexes on flag varieties, $p$-Kazhdan-Lusztig
polynomials describe the stalks of parity sheaves with coefficients in
a field of charateristic $p$ \cite[Part 3]{RW}.

Kazhdan-Lusztig polynomials are characterised by a self-duality
condition and a degree bound, which mimics the defining properties of
the intersection cohomology sheaf. Currently there exists no
similar combinatorial characterisation of $p$-Kazhdan-Lusztig polynomials,
however they always satisfy the self-duality condition. An important
problem concerning $p$-Kazhdan-Lusztig polynomials is whether the
degree bound for $p$-Kazhdan-Luszig polynomials can ``fail by more than
one''.  In the language of parity sheaves, this
translates into the question as to whether an indecomposable
parity sheaf is necessarily perverse. In the language of Soergel bimodules, it
translates into the question as to whether an indecomposable Soergel
bimodule can have non-zero (and necessarily nilpotent) endomorphisms
of negative degree. In this paper an indecomposable Soergel bimodule
posessing such an endomorphism is called \emph{non-perverse}.

Since the beginning of the theory of parity sheaves, it was known that
parity sheaves need not be perverse on nilpotent cones and on the affine flag
variety (see \cite[\S 4.3]{JMW2} and
\cite[Lemma 3.7]{JMW3}). In 2009, the second author found an
example of a parity 
sheaf on a finite flag variety of type $C_3$ in characteristic $2$ which is not perverse
\cite[\S 5.4]{JW}. Moreover, a recent conjecture of Lusztig and the
second author implies that parity sheaves can be arbitrarily far
from being perverse on the affine flag manifold of $\textrm{SL}_3$ \cite{LW}.
However, in \cite{JMW3,MR} it is proved that
parity sheaves on the affine Grassmannian \emph{are} perverse as long as $p$
is a good prime. (Recall that a prime $p$ is good for a
fixed root system if it does not divide any coefficient of the highest
root when expressed in the simple roots.)
Extenstive calculations on finite flag
varieties have suggested that parity sheaves are perhaps 
perverse in good characteristic. Recently, Achar and Riche proved
that this would have nice consequences (existence of ``Koszul like
gradings'') on modular category $\mathcal{O}$ \cite{ARIII}.

In this note we prove the existence of a parity sheaf in characteristic
$p = 2$\footnote{Note that $p = 2$ is good for $\textrm{GL}_n$.} on
the flag variety $\textrm{GL}_{15}/B$ which is not
perverse. In other words, non-perverse Soergel
bimodule for $S_{15}$ exist. 
Our construction is a variation of the
method of \cite{WExplosion}.
We 
expected to be able to produce many examples in this way (and thus
obtain results similar to \cite{WExplosion} for non-perverse Soergel
bimodules), however extensive computer calculations only produced a
few more examples, all in characteristic $2$. 

Another interesting consequence of our construction is that it gives a
Schubert variety for the general linear group with no semi-small
(generalised) Bott-Samelson resolution. More generally, the Schubert
 variety in question admits no semi-small even (in the sense of \cite[\S 2.4]{JMW2})
resolution. Probably the Schubert
variety admits no semi-small resolution at all.

\emph{Acknowledgements:} The second author presented these results at
AIM, Palo Alto in April, 2016. We would like to thank the organisers
for the opportunity to discuss this construction and the participants
(and especially P.~McNamara) for useful comments. We are very grateful
to the Clay Math Institute for funding for a visit of the first author
to Sydney, where this paper was completed. The first author was
supported by Fondecyt 1160152 and Proyecto Anillo ACT 1415 PIA-CONICYT.

\section{Soergel and Singular Soergel Diagrammatics}

\subsection{Hecke algebra and spherical module}
Fix $n \ge 0$. Let $W := S_n$
denote the symmetric group on $n$ letters, viewed as a Coxeter group 
{ $(W,S)$ where  $S = \{ s_i \}_{1 \le i \le n-1}$ is the set of simple transpositions (i.e. $s_i := (i,i+1)$),   with}
  length function $\ell$ and Bruhat order $\le$.
Let $H$ denote the Hecke algebra of $(W,S)$ with standard {$\ZM[v^{\pm 1}]$-}basis $\{
h_x\}_{x \in W}$ and Kazhdan-Lusztig basis $\{ b_x\}_{x
\in W}$ (e.g. $b_s = h_s + vh_{\id}$ for all $s \in S$). We write
$b_x := \sum \b_{y,x} h_y$ (so $\b_{y,x}$ are Kazhdan-Lusztig
polynomials). For any expression $\un{x} := (s_1, s_2, \dots, s_m)$ we
set $b_{\un{x}} := b_{s_1} b_{s_2} \dots b_{s_m}$.

For any subset $A \subset S$ we denote by $W_A$ the (standard parabolic) subgroup it
generates, by $w_A \in W_A$ the longest element and by $W^A$ {  the  } minimal
coset representatives for $W/W_A$. Corresponding to
$A$ we have the spherical (left) module $M$ with its standard basis $\{
m_x\}_{x \in W^A}$ and Kazhdan-Lusztig basis $\{ c_x\}_{x
\in W^A}$ (see \cite[\S 3]{SoeKL}, note however that we work with left
modules throughout). We write $c_x := \sum_{y,x\in W^A} \gamma_{y,x} {   m}_y$ (so
$\gamma_{y,x}$ are spherical Kazhdan-Lusztig polynomials). The map
$m_x \mapsto h_{xw_A}$ gives an embedding $\phi : M \into H$ of left
$H$-modules mapping $c_x \mapsto b_{xw_A}$ \cite[Proposition 3.4]{SoeKL}. Given any
expression $\un{x} := (s_1, s_2, \dots, s_m)$ we
set $c_{\un{x}} := b_{s_1} b_{s_2} \dots b_{s_m} \cdot m_{\id}\in M$.

\subsection{Diagrammatic Soergel bimodules}
Fix a field $\Bbbk$ of characteristic $p \ge 0$ and let $R :=
\Bbbk[x_1,\dots, x_n]$ be graded with $\deg x_i = 2$. The symmetric group
$W$ acts naturally on $R$ via permutation of variables.
To $W$ and $R$ one may associate an additive graded and
Karoubian monoidal category $\HC$ of ``diagrammatic Soergel
bimodules'' as in \cite{EW}. We denote the shift functor on $\HC$ by
$B \mapsto B({   i})$ for ${   i} \in \ZM$. For any expression
$\un{x}$ we denote by $B_{\un{x}}$ the corresponding Bott-Samelson
object in $\HC$ and (if $\un{x}$ is reduced) by $B_x$ its maximal
indecomposable summand. By \cite[Theorem
6.26]{EW} the set $\{ B_x \; | \; x \in W \}$ is a complete set of
isomorphism classes of indecomposable objects in $\HC$, up to shift
and isomorphism. We denote by $[\HC]$ the split Grothendieck {   ring}
of $\HC$ (a $\ZM[v^{\pm 1}]$-module via $v \cdot [M] := [M(1)]$) and by
$\ch : \HC \to H$ the character \cite[\S 6.5]{EW} ($\ch$ is uniquely
characterised by $\ch(B_{\un{x}}) = b_{\un{x}}$ and $\ch(M(1)) = v\ch(M)$). It induces an
isomorphism $\ch : [\HC] \simto H$.
We denote by $\p b_x := \ch(B_x) \in H$
the character of $B_x$. The set $\{ \p b_x\}_{x \in W}$ only depends on the
characterstic $p$ of $\Bbbk$ and yields the \emph{$p$-canonical basis} 
of $H$. We have ${}^0 b_x = b_x$ for all $x \in W$.  The expression $\p b_x = \sum \p \beta_{y,x} h_y$ defines the
\emph{$p$-Kazhdan-Lusztig polynomials} $\p \beta_{y,x}$.

\subsection{Diagrammatic Spherical category}
For any subset $A \subset S$ we denote by $\MC$ the spherical category
associated to $A$ (see \cite[\S 5]{EThick}). It is a graded additive
Karoubian left $\HC$-module category with shift functor $C \mapsto C(m)$. We denote by $C_{\id}$ the
``identity'' of $\MC$ (in the notation of \cite{EThick}, $C_{\id}$ is given by
the empty diagram consisting only of the $A$-colored membrane).
For any expression $\un{x}$ we
denote by $C_{\un{x}}$ the object $B_{\un{x}} \cdot C_{\id}$ and (if
$\un{x}$ is a reduced expression for $x \in W^A$) by $C_x$ its maximal
indecomposable summand. The set $\{ C_x \; | \; x \in W^A \}$ is a complete set of
isomorphism classes of indecomposable objects in $\MC$, up to shift
and isomorphism. We denote by $[\MC]$ the split Grothendieck group of
$\MC$ (a $\ZM[v^{\pm 1}]$-module as above) and by
$\ch : \MC \to M$ the character (it is uniquely
characterised by $\ch(C_{\un{x}}) = {   c}_{\un{x}}$ and $\ch(C(1)) =
v\ch(C)$).  The character map satisfies $\ch(BC)
= \ch(B)\ch(C)$ for all $B \in \HC$ and $C \in \MC$, and induces an
isomorphism $\ch : [\MC] \simto M$ of left $H {   \cong} [\HC]$-modules.
We denote by $\p c_x := \ch(C_x) \in M$ { the character of $C_x$}. The set $\{ \p c_x\}_{x \in W^A}$ yields the \emph{$p$-canonical basis}
of $M$. We have ${}^0 c_x = c_x$ for all $x \in W^A$.  The expression
$\p c_x = \sum_{y \in W^A} \p \gamma_{y,x} m_y$ defines the
\emph{spherical $p$-Kazhdan-Lusztig polynomials} $\p \gamma_{y,x}$.

There is a functor $\Phi : \MC \to \HC$ of left $\HC$-module categories which
sends $C_{\id}$ to $B_{w_A}$ (see \cite[Definition 5.4]{EThick}) and
satisfies $\Phi(C_x) =
B_{xw_A}$ for all $x \in W^A$. 
Passing to split Grothendieck groups as above it realises
the embedding $\phi : M \into H$. In particular $\phi(\p c_x) = \p b_{xw_A}$ and
hence 
\begin{equation} \label{eq:MHKL}
\p \gamma_{y,x} = \p \beta_{yw_A, xw_A}
\end{equation}
for all $x, y \in W^A$.
\subsection{Soergel's hom formulas} 
 Consider the
bilinear form $( -, - ) : H \times H \to \ZM[v^{\pm 1}]$ on $H$
defined in \cite[\S 2.4]{EW}. It satisfies $( p h,  q h' ) = \overline{p} q (h, h')$, 
$(b_sh, h') = (h, b_sh')$ and $(hb_s, h') = (h, h'b_s)$  for all $p, q
\in \ZM[v^{\pm 1}]$, $h, h' \in H$ and $s \in S$ (see \cite[\S
2.4]{EW}). Similarly, there is a unique bilinear form $( -, - ) : M
\times M \to \ZM[v^{\pm 1}]$ defined by $(m,m') := (\phi(m),
\phi(m'))/\widetilde{\pi}(A) \in \ZM[v^{\pm 1}]$, where $\widetilde{\pi}(A) := \sum_{{x} \in
  W_A} v^{2\ell(x)}$. It satisfies
$( p m,  q m' ) = \overline{p} q (m, m')$ and 
$(b_sm, m') = (m, b_sm')$ for all $p, q \in \ZM[v^{\pm 1}]$, $m, m'
\in M$ and $s \in S$. (It is not immediately obvious that $(\phi(m),
\phi(m'))/\widetilde{\pi}(A)$ always belongs to $\ZM[v^{\pm 1}]$, but
this is the case by \cite[(2.9)]{WSing}.)

Given $B, B' \in \HC$ we denote by $\gHom(B,B')
:= \bigoplus_{n \in \ZM} \Hom_{\HC}(B,B'(n))$, which is naturally a graded
$R$-bimodule. Similarly, given $C, C' \in \MC$ we denote {by} $\gHom(C,C')
:= \bigoplus_{n \in \ZM} \Hom_{\MC}(C,C'(n))$, which is naturally a graded
$(R,R^A)$-bimodule (see \cite[Definition 5.1]{EThick}). \emph{Soergel's hom formulas}  
(crucial below) are the statements:
\begin{gather}
  \label{eq:1}
\begin{array}{c}
 \text{For $B,B' \in \HC$, $\gHom(B,B')$ is graded free as a left
   $R$-module,} \\ 
 \text{of graded rank $( \ch(B), \ch(B') )$.}
\end{array} \\
  \label{eq:2}
\begin{array}{c}
 \text{For $C,C' \in \MC$, $\gHom(C,C')$ is graded free as a left
   $R$-module,} \\ 
 \text{of graded rank $( \ch(C), \ch(C') )$.}
\end{array}
\end{gather}

As in the introduction, we say that an indecomposable object $X \in
\HC$ (resp. $X \in \MC$) is \emph{perverse} if it has no non-zero
endomorphisms of negative degree. (This terminology comes from
\cite{EW2}, where such bimodules play a key role in the proof of
Soergel's conjecture.) The following lemmas are a direct consequence
of the hom formulas above (see \cite[(6.1)]{EW2}):
\begin{lem}
A self-dual Soergel bimodule $B$ is perverse $\iff $  $\mathrm{ch}( B )=\bigoplus_{z\in W}\mathbb{Z}b_z. $ 
\end{lem}

\begin{lem}\label{Mpositive}
A self-dual element ${C}\in \mathcal{M}$ is perverse $\iff $  $\mathrm{ch}({C})=\bigoplus_{z\in W^A}\mathbb{Z}c_z. $ 
\end{lem}

Below we will prove that there exists a non-perverse object in $\MC$.
By \eqref{eq:MHKL} non-perverse objects in $\MC$ produce non-perverse
objects in $\HC$ by application of the functor $\Phi$.

\subsection{Intersection forms in $\mathcal{M}$} In what follows we
identify $W^A$ and $W/W_A$ via the canonical isomorphism given by the
composition $W^A \into W \onto W/W_A$.
If $I \subset W/W_A$ is an ideal (i.e. $x \le y \in I \Rightarrow x \in
I$)  we denote by $\mathcal{M}_I$ the ideal of $\mathcal{M}$
generated by all morphisms which factor through an object 
$C_{\un{y}}$, for any reduced expression $\un{y}$ for $y
\in I$.

Given $x \in W^A$ we denote by $\mathcal{M}^{\ge x}$ the quotient category
$\mathcal{M} / \mathcal{M}_{\not \ge x}$ where $\not\ge~\hspace{-0.25cm} x := \{ y \in W^A\;|\; y \not \ge
x \}$. We write $\Hom_{\ge x}$ for (degree zero) morphisms in
$\mathcal{M}^{\ge x}$. All objects  $C_{\un{x}}$
corresponding to 
reduced expressions $\un{x}$ for $x$ become canonically isomorphic to
$C_x$ in $\mathcal{M}^{\ge x}$. For $C \in \MC$ and any $x \in W^A$
the spaces
\[
\Hom^\bullet_{\ge x}( C_x, M)  = \bigoplus \Hom_{\ge x}(C_x, M(i)) \quad
\text{and} \quad
\Hom^\bullet_{\ge x}( M, C_x)  = \bigoplus \Hom_{\ge x}(M,C_x(i))
\]
are free as graded left $R$-modules of graded rank $p_x$ where $\ch(M) = \sum_{x \in W^A} p_x
m_x$. In particular, we have $\End_{\ge x}(C_x) = R$.
Given an expression $\un{w}$ and an element $x\in  W^A,$ the \emph{intersection form}
is the canonical pairing
\[
I_{x,\un{w},d}^{\Bbbk} : \Hom_{\ge x}(C_x(d), C_{\un{w}}) \times \Hom_{\ge x}(C_{\un{w}}, C_x(d))
\to \Bbbk = \End_{\ge x}(C_x(d))/(R^+)
\]
where $R^+ \subset R$ denotes the ideal of elements of positive degree.

\begin{lem} \label{lem:rank}
  The multiplicity of $C_x(d)$ as a summand of $C_{\un{w}}$ equals
  the rank of $I^{\Bbbk}_{x,\un{w},d}$.
\end{lem}
\begin{proof} This claim is standard for a   $\Bbbk$-linear  Krull-Schmidt category with finite dimensional Hom spaces. For one proof see  \cite[Lemma 3.1]{JMW2}.
\end{proof}

\subsection{Parabolic defect}
Fix a word $\un{y} = s_{i_1} \dots s_{i_m}$ in $S$ representing an element $y\in W$. A
\emph{subexpression} of $\un{y}$ is a sequence $\un{e} = e_1 \dots e_m$ with
$e_i \in \{ 0, 1 \}$ for all $i$. We set $\un{y}^{\un{e}} :=
s_{i_1}^{e_1} \dots s_{i_m}^{e_m} \in W$. Any subexpression $\un{e}$
determines a sequence $y_0, y_1, \dots, y_m \in W$ via $y_0 := \id{}$, $y_j
:= s_{i_{m+1-j}}^{e_{m+1-j}}y_{j-1}$ for $1 \le j \le m$ (so $y_m =
\un{y}^{\un{e}}$). Given a subexpression $\un{e}$ we associate a
sequence $d_j \in \{ U, D, S \}$ (for \emph{U}p, \emph{D}own, \emph{S}tay) via
\[
d_j := \begin{cases} U & \text{if $ y_{m-j}<s_{i_j} y_{m-j}\in W/W_A$,}\\
D & \text{if $ y_{m-j}>s_{i_j}y_{m-j}\in W/W_A,$}\\
S& \text{if $s_{i_j}y_{m-j} = y_{m-j}$ in $W/W_A$.} \end{cases}
\]
We usually view $\un{e}$ as the decorated sequence $( d_1e_1, \dots,
d_me_m)$. The
\emph{parabolic defect} of $\un{e}$ is
\[
\mathrm{pdf}(\un{e}) := | \{ i \; | \; d_ie_i = U0\ \mathrm{or} \ S1 \}| - | \{ i \; | \; d_ie_i = D0 \ \mathrm{or} \ S0 \}|.
\]

One can see readily from the formula for $b_sc_x$ (see \cite[\S 3]{SoeKL} for example) the following formula of Dehodar 
\begin{equation}\label{Dehodar}
m_{\underline{x}}=\sum_{\underline{e}\subset \underline{x}}v^{\mathrm{pdf}(\underline{e})}m_{\underline{x}^{\underline{e}}}.\end{equation}
where  $\underline{x}^{\underline{e}}$ is viewed as an element of $W/W_A.$

\section{Existence of a non-perverse indecomposable Soergel bimodule}

\subsection{Strategy of the proof}
Consider a reduced expression $\underline{w}$ representing an element $w\in W^A$ and another element $x\in  W^A,$ such that 
\begin{gather}\label{intersection}
\mathrm{rk}(I^{\mathbb{Q}}_{x,\un{w},-1})=1\ ,\
\mathrm{rk}(I^{\mathbb{F}_2}_{x,\un{w},-1})=0  \ \ \mathrm{and} \\
\label{interval}
m_{\underline{w}}\in \bigoplus_{x<z\leq w}\mathbb{Z}[v]m_z\oplus \bigoplus_{x\nless z}\mathbb{Z}[v, v^{-1}]m_z.
\end{gather}

\begin{lem} If the  above hypotheses  are satisfied, then there is a
  non-perverse indecomposable object in $\mathcal{M}$ over the field
  $\mathbb{F}_2$.
\end{lem}
\begin{proof}
Suppose, for contradiction, that there is no non-perverse
indecomposable object in $\mathcal{M}$ over the field
$\mathbb{F}_2$. By Lemma \ref{Mpositive} we
have  \begin{equation}\label{positive}
{}^2\underline{m}_y\in \bigoplus_{z\in W^A}\mathbb{Z}\underline{m}_z\
\ \ \mathrm{for\ all \ }y<w. 
\end{equation}
By Lemma \ref{lem:rank} we have the following formulae 
\begin{gather*}
c_{w}=m_{\underline{w}}-\sum_{\substack{x<w \\
x\in {}^aW}} \bigg( \sum_{d\in \mathbb{Z}}
\mathrm{rk}(I^{\mathbb{Q}}_{x,\un{w},d})v^d \bigg)  c_x
\quad \text{and} \quad
{}^2c_{w}=m_{\underline{w}}-\sum_{\substack{x<w \\
x\in {}^aW}} \bigg( \sum_{d\in \mathbb{Z}} \mathrm{rk}(I^{\mathbb{F}_2}_{x,\un{w},d})v^d \bigg)  {}^2c_x.
\end{gather*}
By Equations (\ref{interval}) and (\ref{positive}),  we have that
$v^{-1}m_x$ appears in the expansion of   ${}^2m_{\underline{w}}$ in
the standard basis. Now apply Lemma \ref{Mpositive} again.
\end{proof}

\subsection{The elements $\underline{w}$ and $x$ }

Consider the string diagram in Figure \ref{wpic}.
We denote the corresponding reduced expression (obtained by reading
from bottom to top) as $\un{w}$, so that $$\un{w} = (s_1, s_2, \ldots, s_{14},  s_2, s_3, \dots, s_{13}, s_4, s_5, \ldots, s_{12},s_3,s_4,\ldots, s_{11}, \ldots )=:(t_1,\ldots, t_{78}).$$ 
 Let $w\in W^A$ be the element represented by
 $\underline{w}$. By  \cite[Lemma
 5.5]{WExplosion} (or by simple inspection), $\underline{w}$ is
 reduced.  Let us define $x:=w_B\in W^A.$

\begin{figure}
\centering
\includegraphics[width=0.22\textwidth]{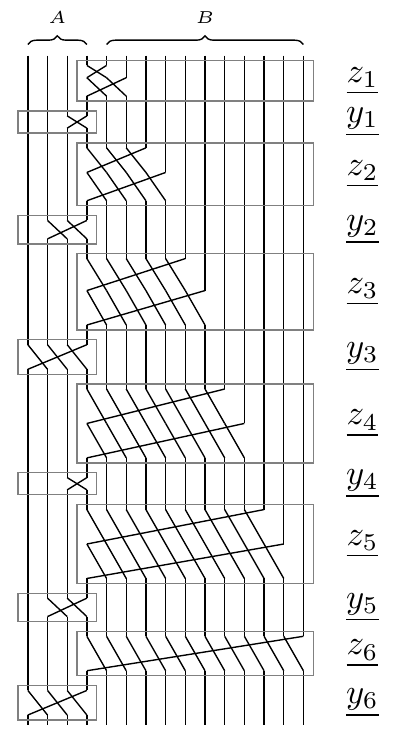}
\centering
\caption{The reduced expression $\un{w}$.}
\label{wpic}
\end{figure}

\subsection{Proof of equations \ref{intersection} and  \ref{interval}}

 If $\underline{e}=(e_1,e_2, \ldots, e_{78})$ is a reduced expression of $\underline{w}$ with $\underline{w}^{\underline{e}}\in w_BW_A$, then, by \cite[Lemma
 5.6]{WExplosion} $t_i=s_4\implies e_i=0$ and  for length reasons one
 has that if $t_i=s_{j}$ with $j\geq 5$ then $e_i=1.$ In both these cases $d_i$
 is $U$.  So a subexpression $\underline{e}$ of $\underline{w}$ for
 $x$ is completely determined by its ``$y$-part" (see Figure \ref{wpic}). In other words, it
 is determined by the sets  $I^0:=\{i\, \vert \,
 t_i=s_{j}\text{ with}\ j\leq 3\text{ and }e_i=0\}$ and $I^1:=\{i\, \vert \,
 t_i=s_{j}\text{ with}\ j\leq 3\text{ and }e_i=1\}$ (in these cases $d_i=S$). By the definition of the
 parabolic defect, one has that 
 $$\mathrm{pdf}(\underline{e})=11-|I^0 |+|I^1|.$$
With this formula we see that if $\underline{e}$ is a subexpression of $\underline{w}$ for $x$ we have 
$\mathrm{pdf}(\underline{e})=-1 \iff |I^0|=12$, and thus
there is only one subexpression satisfying this. On the other hand,
$\mathrm{pdf}(\underline{e})=1 \iff |I^0|=11$, thus there
are $12$ subexpressions satisfying this. Thus, in this case, the
intersection form $I^{\mathbb{Q}}_{x,\un{w},-1}$ is a $1\times
12$-matrix. One can calculate explicitly that this matrix is given
by $$(-2, -2, 0, -2, -2, 0, -2, -2, -2, 2, 0, 0).$$
To perform this calculation one uses the main result of
\cite{HW} together with the same reductions used at the end of
\cite[\S 5]{WExplosion}. For $1 \le i < 15$ let $\alpha_i := x_{i+1} -
x_i \in R$ denote the simple root and let $\partial_i : R \to R(-2) :
f \mapsto (f - s_i(f))/\alpha_i$ denote the
Demazure operator. Each of the $12$
entries above is the result of erasing one $\partial_i$ from the following expression (which is equal to $0$ for degree reasons) 
$$\partial_1\partial_2\partial_3 (\alpha_4   \partial_2\partial_3    (\alpha_4^2   \partial_3    (\alpha_4^2   \partial_1\partial_2\partial_3  (\alpha_4^2     \partial_2\partial_3  (\alpha_4^2     \partial_3(\alpha_4^2)))))) .$$
For example, if we erase the fourth $\partial$ (i.e.  $\partial_2$),
we obtain the fourth entry of the intersection form 
$$\partial_1\partial_2\partial_3 (\alpha_4   \partial_3
(\alpha_4^2   \partial_3
(\alpha_4^2   \partial_1\partial_2\partial_3
(\alpha_4^2     \partial_2\partial_3
(\alpha_4^2     \partial_3(\alpha_4^2))))))  = -2.$$
Note that this matrix has rank 1 over $\QM$, but rank $0$ over a field
of characteristic $2$. This proves \eqref{intersection}.

To check \eqref{interval} one needs to do a big computer check. But
there is one point that needs explanation about this calculation. If
one considers all the subexpressions of $\underline{w}$ one has $2^{78}$
possibilities. This is too big even for our computer! Suppose
$\underline{e}=(e_1, e_2, \ldots, e_{78})$ 
is a subexpression such that $\un{w}^{\un{e}}W_A$ belongs to the the
interval $[xW_A, wW_A]$, then for any $i$ such that
$t_i\in W_B$ we must have $e_i=1$. Thus we ``only'' need to check $2^{23}$
subexpressions, which is feasible by computer (and takes our machine 20
minutes). 

\def\cprime{$'$} \def\cprime{$'$} \def\cprime{$'$}


\end{document}